\UseRawInputEncoding
 \documentclass[11pt]{amsart}
\usepackage{amsmath}
\usepackage{hyperref} 
\usepackage{amsfonts}
\usepackage{amssymb}
\usepackage{amsthm}
\usepackage[arrow, matrix, curve]{xy} 
\usepackage{hyperref}
\usepackage{amssymb,amscd,url,tikz,stmaryrd,mathtools,fixltx2e}
\usepackage{amssymb,amscd,url}

\usepackage{color}

\newcommand{\Z}{{\mathbb Z}}
\renewcommand{\S}{{\mathbb S}}
\newcommand{\R}{{\mathbb R}}
\newcommand{\C}{{\mathbb C}}
\newcommand{\D}{{\mathbb D}}
\newcommand{\N}{{\mathbb N}}

\def\SL{{\rm SL}}

\def\z{\overline{z}}
\renewcommand{\matrix}[1]{\left(\begin{array}{cc} #1\end{array}\right)}

\newcommand{\wt}[1]{\widetilde{#1}}
\newcommand{\wh}[1]{\widehat{#1}}

\theoremstyle{plain}
\newtheorem{theorem}{Theorem}
\newtheorem{lemma}{Lemma}
\newtheorem{proposition}[lemma]{Proposition}
\newtheorem{remark}[lemma]{Remark}

\newtheorem{definition}[lemma]{Definition}

\DeclareFontFamily{U}{mathx}{\hyphenchar\font45}
\DeclareFontShape{U}{mathx}{m}{n}{
      <5> <6> <7> <8> <9> <10>
      <10.95> <12> <14.4> <17.28> <20.74> <24.88>
      mathx10
      }{}
\DeclareSymbolFont{mathx}{U}{mathx}{m}{n}
\DeclareFontSubstitution{U}{mathx}{m}{n}
\DeclareMathAccent{\widecheck}{0}{mathx}{"71}
\DeclareMathAccent{\widetilde}{0}{mathx}{"72}
\DeclareMathAccent{\widebar}{0}{mathx}{"73}
\DeclareMathAccent{\widevec}{0}{mathx}{"74}
\DeclareMathAccent{\widehat}{0}{mathx}{"70}
\DeclareMathAccent{\widefrown}{0}{mathx}{"75}
\DeclareMathAccent{\chinesehat}{0}{mathx}{"69}

\def\SL{{\rm SL}}

\def\z{\overline{z}}
\renewcommand{\matrix}[1]{\left(\begin{array}{cc} #1\end{array}\right)}

\begin{document}

\title{Fuchsian DPW potentials for Lawson surfaces }

 \author{Lynn Heller}
 \author{Sebastian Heller}

\address{Institut f\"ur Differentialgeometrie\\
Welfengarten 1\\
30167 Hannover\\
Germany} 
\email{lynn.heller@math.uni-hannover.de}
 \noindent 
\address{Institut f\"ur Differentialgeometrie\\
Welfengarten 1\\
30167 Hannover\\
Germany} 
\email{seb.heller@gmail.com}

\begin{abstract}
The Lawson surfaces $\xi_{1,g}$ of genus $g$ are constructed by rotating and reflecting the Plateau solution $f_t$ with respect to a particular geodesic $4$-gon $\Gamma_t$ along its boundary, where $t= \tfrac{1}{2g+2}$ is an angle of  $\Gamma_t$.  In this paper we combine the existence and regularity of the Plateau solution $f_t$ in $t \in (0, \tfrac{1}{4})$ with topological information about the moduli space of Fuchsian systems on the 4-puncture sphere to obtain existence of a Fuchsian DPW potential $\eta_t$ for every $f_t$ with $t\in(0, \tfrac{1}{4}]$. Moreover, the coefficients of $\eta_t$ are shown to depend real analytically on $t$. This implies that the Taylor approximation of the DPW potential $\eta_t$ and of the area obtained at $t=0$ found in \cite{HHT2} determines these quantities for  all $\xi_{1,g}$. In particular, this leads to an algorithm to conformally parametrize all Lawson surfaces $\xi_{1,g}$.

\end{abstract}

\maketitle

\section*{Introduction}\label{sec:intro}

In 1970 Lawson \cite{Lawson} constructed embedded minimal surfaces of every genus in the round 3-sphere.  These are obtained from the Plateau solution of a geodesic polygon which is then reflected and rotated along its boundary. Using similar philosophy, Karcher-Pinkall-Sterling \cite{KPS} found minimal surfaces in $S^3$ with Platonic symmetries. Later Kapouleas and Yang \cite{KapYan, Kap} constructed high genus minimal surfaces by doubling a geodesic sphere. 

 Kusner \cite{Kusner} conjectured that the simplest Lawson surfaces $\xi_{1,g}$ minimize the Willmore energy among  immersions from a genus $g$ surface in generalization of the famous Willmore conjecture solved by Marques and Neves \cite{MN}. Though few examples of compact minimal surfaces of  a given genus $g\geq2$ are known, this conjecture is supported by computer experiments \cite{Kusner} where an arbitrary compact surface in the 3-sphere is deformed via an energy decreasing flow and converges to a shape resembling a Lawson surface. The Willmore energy of $\xi_{1,g}$ is strictly below $8\pi$ but converges to  $8\pi$ in the genus limit  see \cite{Kusner}. This coincides with the large genus energy limit of genus $g$ Willmore minimizers \cite{KLS}.  Moreover, the stability properties of these candidates, though as minimal surfaces rather than as Willmore surfaces,  was studied in \cite{KapWiyStability}. Due to the implicit way the Lawson surfaces are constructed, determining the geometric properties, such as computing their area (or Willmore energy), is very difficult though
in view of the Kusner conjecture very desirable.

In an effort to obtain more explicitness, a completely different approach to constructing minimal (and constant mean curvature) surfaces in the 3-sphere was taken in \cite{HHS, HHT, HHT2}.  In particular, the Lawson surfaces $\xi_{1,g},$ for large $g,$ have been reconstructed (i.e., independent from the Lawson's result)  via an implicit function theorem argument using methods from integrable systems. The method itself can be interpreted as a global version of the Weierstra\ss $\;$representation, which is also often referred to as DPW approach \cite{DPW} in this context. For tori the approach was pioneered by Hitchin \cite{Hitchin} and Pinkall-Sterling \cite{PS} around 1990, and Bobenko \cite{Bobenko} gave an explicit parametrization of all CMC tori in 3-dimensional space forms. \\

Consider hereby a conformally parametrized minimal immersion $f$ from a compact genus $g$ Riemann surface $M_g$ into the round $3$-sphere. Then, $f$ is harmonic giving rise to a symmetry of the Gauss-Codazzi equations  inducing an associated $\S^1$-family of (isometric) minimal surfaces on the universal covering of $M_g$ with rotated Hopf differential. The gauge theoretic counterpart of this symmetry is manifested in an associated $\C^*$-family 
of flat $\text{SL}(2,\mathbb C)$-connections $\nabla^\lambda$  \cite{Hitchin} on the trivial $\C^2$-bundle over $M_g$ solving the following {\em Monodromy Problem} 
\begin{enumerate}\label{closingconditions}
\item[(i)] conformality:  $\nabla^\lambda=\lambda^{-1}\Phi+\nabla-\lambda \Psi$ for a nilpotent $\Phi\in \Omega^{1,0}(M_g,\mathfrak{sl}(2,\mathbb C)) ;$
\item[(ii)]  intrinsic closing: $\nabla^\lambda$ is unitary for all $\lambda\in\S^1,$ i.e., $\nabla$ is unitary and $\Psi=\Phi^*$ with respect to the standard hermitian metric on $\underline{\C}^2$;
\item[(iii)] extrinsic closing: $\nabla^\lambda$ is trivial for $\lambda=\pm1.$
\end{enumerate}
The minimal surface can be reconstructed from the associated family of connections as the gauge between
$\nabla^{-1}$ and $\nabla^1.$ Constructing minimal surfaces is thus equivalent to writing down appropriate families of flat connections. 
The DPW method \cite{DPW} is a way to generate such families of flat connections on a Riemann surface from so-called {\em DPW potentials}, denoted by  $\eta = \eta^\lambda,$ $ \lambda \in \C^*,$ 
using loop group factorisation. In fact,   $\eta^\lambda$ fixes the gauge class of the connections $\nabla^\lambda$ as
$$d+ \eta^\lambda \in [\nabla^\lambda].$$ On simply connected domains $\mathbb U$, all DPW potentials give rise to minimal surfaces from $\mathbb U$. Whenever the domain has non-trivial topology, finding DPW potentials satisfying  
conditions equivalent to (i)-(iii) becomes difficult. 
 
Though successful in the case of tori, the first embedded and closed minimal surfaces of genus $g>1$ using DPW were only recently constructed in \cite{HHT}. This is due to the fact that, in contrast to tori, the fundamental group of a higher genus surface is non-abelian. A global version of DPW has been developed in \cite{He1, He2} under certain symmetry assumptions. The main challenge to actually construct higher genus minimal and CMC surfaces is to determine infinitely many parameters in the holomorphic ``Weierstra\ss''-data.

For large genus, we succeeded in computing these parameters via implicit function theorem, leading to an alternate existence proof of the Lawson surfaces $\xi_{1,g}$ in \cite{HHT, HHT2}. More explicitly, we construct families $f^t_\varphi$ of minimal and CMC surfaces, for $t\sim0$ and $\varphi \in (0, \tfrac{\pi}{2})$, starting at two geodesic spheres intersecting at angle $2 \varphi$ and deform its DPW potential in direction of the corresponding Scherk surface in $\R^3$ such that $f^t_{\tfrac{\pi}{4}}= \xi_{1,g}$ at $t= \tfrac{1}{2(g+1)}$. 
Through the implicit function theorem, we give in \cite{HHT2} an iterative algorithm to compute the Taylor expansion of the DPW-potentials $\eta^t_\varphi$ and the area Area$_\varphi(t)$ of $f^t_\varphi$ at $t=0.$ In particular, we obtain for $\varphi = \tfrac{\pi}{4}$ 
\begin{equation}\label{Area}
\text{Area}(f^t_{\tfrac{\pi}{4}}) \sim 8 \pi \left(1 - \log(2) t - \tfrac{9}{4}\zeta (3) t^3 + O(t^5) \right),
\end{equation}
where $\zeta$ is the Riemann $\zeta$-function. By the regularity statement of the implicit function theorem, the family of DPW-potentials $\eta^t_\varphi$ as well as the area Area$_\varphi(t)$ depends real analytically on the parameter $t \sim 0$. 

This paper is about quantitative results concerning the existence interval of the solutions $\eta_t := \eta^t_{\tfrac{\pi}{4}}.$ The idea is to use the properties of the Plateau solutions for all $t\in (0, \tfrac{1}{4}]$ to prove existence of the family DPW-potentials $\eta_t$ found in \cite{HHT2} on the same time interval. This covers all Lawson surfaces $\xi_{1,g}$. As a byproduct we obtain the real analyticity of $\eta_t$ and Area$_{\tfrac{\pi}{4}}(t)$ for all $t\in (0, \tfrac{1}{4}]$. Together with \cite{HHT2} this leads to an algorithm to computing the area and an explicit conformal parameterisation of $\xi_{1,g}$ for every genus by computing their Taylor expansions at $t=0.$ For every coefficient this involves solving a finite dimensional linear system in terms of multi-polylogarithms. In contrast to \cite{HHT, HHT2} and \cite{HHS} the existence result for the DPW potential in this paper relies on the existence and regularity of the Plateau solutions.

\subsection*{Acknowledgements}
We thank Reiner Sch\"atzle for providing the ideas to prove Theorem \ref{realanalyticity}.  Moreover, we thank Martin Traizet for various fruitful discussions.
The authors  are supported by the  {\em Deutsche Forschungsgemeinschaft} within the priority program {\em Geometry at Infinity}.

\section{Lawson surfaces revisited}\label{Lawson}  Consider  $\mathbb S^3\subset\C^2$ and the four points in $\mathbb S^3$ given by
  \[P_1=(1,0), \quad Q_1=(0,1),\quad P_2=(i,0),\quad Q_2=(0,e^{2\pi i t}).\]
  Let  $f= f_t \colon \mathbb D \rightarrow \mathbb S^3$ be the Plateau solution 
with respect to the closed geodesic 4-gon 
 \[\Gamma=\Gamma_t= P_1 Q_1P_2 Q_2,\]
where $\mathbb D$  is the closed disc of radius $1$ in $\C$ and $f(\partial \D) = \Gamma.$
As $f$ is immersed except at the points $P_1, P_2$ and $Q_1,$ $Q_2,$ we can consider the induced Riemann surface structure on 
 \[\mathring{\mathbb D} :=\mathbb D\setminus f^{-1}\{P_1,Q_1,P_2,Q_2\}.\]
 Note that the boundary of $\mathring{\mathbb D}$ consists of four connected components. Moreover, let $\mathcal G$ denote the group of  automorphisms generated by rotations by $e^{2\pi t i}$ in the $0\oplus\C$-plane.  The Plateau solution extends by reflections across the boundaries to a complete minimal surface 
 $$f^t \colon M_t \rightarrow S^3.$$ 
 The group $\mathcal G$ can  be considered as symmetries of the minimal surface by uniqueness of Plateau solution. The constructed complete minimal surface $M_t,$ is compact if and only if $t$ is rational and embedded of genus $g$ if $t= \tfrac{1}{2g+2}$. Many geometric properties can be derived from rational $t$ from the compactness of $M_t$. With the following theorem these properties carry over to all $t \in (0, \tfrac{1}{4}]$ using continuity arguments.
 
   \begin{theorem}\label{realanalyticity}
 The Plateau solution $f_t$ depends (up to reparametrization) real analytically on the angle $t$ for $t \in(0,\tfrac{1}{4}]$
 \end{theorem}
 
 \begin{proof}
For  $t\sim 0$ and $t\sim \tfrac{1}{4}$ real analyticity of the Plateau solution in $t$ follows from the implicit function theorem and the real analyticity of the monodromy problem solved in  \cite{HHT} and \cite{HHS}, respectively.

For $t \in (0, \tfrac{1}{4})$ consider the Banach space $W^{2,2}_t(\mathbb D, S^3)$ of  maps from $\mathbb D$ to $S^3$ such that the boundary $\partial \mathbb D$ is mapped to $\Gamma_t$.
We want to show that the family of Plateau solutions $f_t$  is real analytic in $t$ at every $t_0\in (0, \tfrac{1}{4})$.

 Consider the linear map 
$$\Psi_t \colon S^3\subset \R^4 \rightarrow S^3\subset \R^4; \quad \Psi_t= \begin{pmatrix}1&0&0&0\\0&1&0&0\\0&0&1 &-\tfrac{1}{\tan(2\pi  t_0)}+ \frac{\cos(2\pi t)}{\sin(2\pi t_0)}\\0&0&0&\frac{\sin(2\pi t)}{\sin(2\pi t_0)}\end{pmatrix}$$ 

Then $\Psi_t(\Gamma_{t_0}) = \Gamma_t$ and we obtain a bijection between $W^{2,2}_t(\mathbb D, S^3)$ and $W^{2,2}(\mathbb D):=W^{2,2}_{t_0}(\mathbb D, S^3)$ which depends real analytically on $t.$

For $t \sim  t_0 $, all solutions of the Euler-Lagrange equation with boundary $\Gamma_t$ which are $W^{2,2}(\mathbb D)$-close to $f_{t_0}$ can be classified via the implicit function theorem using the strict stability of Plateau solutions of $\Gamma_{t_0}$ by  \cite[Theorem 1]{FCS}, see also \cite[Page 20, proof Lemma 6.7 (ii)]{KapWiyStability} and the fact that the fundamental piece is graphical \cite[Lemma 4.7]{KapWiyStability}. This gives rise to a unique real analytic family $(\wt f_t)_{t \sim t_0}$ of minimal surfaces of disc type with boundary $\Gamma_t$ and $\wt f_{t_0} = f_{t_0}$ which are W$^{2,2}(\mathbb D)$-close to $f_{t_0}$. These solutions $\wt f_t$ must then coincide (up to reparametrization) with the Plateau solutions $f_t$ for all $t \sim t_0$  due to the uniqueness of minimal discs with boundary $\Gamma_t$ \cite[Theorem 4.1]{KapWiyStability} or \cite{Lawson}.

 \end{proof}

 \subsection{Riemann surface structures }
 
Let $\overline{\mathring{\mathbb D}}$ denote the complex conjugate Riemann surface of $\mathbb D$. Define the Riemann surface
 \[\underline\Sigma=\mathring{\mathbb D_1}\cup_\sim \overline{\mathring{\mathbb D_1}}\]
 glued together along the four boundary components via reflections across the four (geodesic) edges
 \[P_1Q_1,\; Q_1P_2,\; P_2Q_2,\; Q_2P_1.\]
 Note that the minimal surface into $\mathbb S^3$ is not  
 {well-defined as a map from $\underline\Sigma$.
 \begin{proposition}\label{pro:cftsigma}
 The Riemann surface $\underline\Sigma$ is (biholomorphic to)
 \[\C P^1\setminus\{\infty, 0,1, -1\}.\]
 \end{proposition}
 \begin{proof}
By construction $\underline\Sigma$ has the topology of sphere
 \[\C P^1=\underline{\Sigma}\cup\{Q_1, Q_2, P_1, P_2\}\]
 with 4 points removed. The aim is to show that the Riemann surface structure extends through the punctures  for all $t \in (0, \tfrac{1}{4}]$. Since the angle at the vertices $Q_1$ and $Q_2$ is $\tfrac{\pi}{2}$, the Riemann surface structure around $Q_1$ and $Q_2$ is  by Lawson \cite[Theorem 1]{Lawson} the Riemann surface structure obtained 
 at the quotient of a
 non-trivial holomorphic $\Z_2$-action.

At $P_1$ and $P_2$  with rational angles $t $,  the argument works analogously to  the proof of \cite[Theorem 1]{Lawson}, showing that the Riemann surface   extends to $P_1$ and $P_2$.

Together this yields for rational $t$ a well-defined compact, generally branched, minimal surface $f_t \colon M_t \rightarrow S^3$ with induced Riemann surface structure on $M_t.$ The Riemann surface $\underline\Sigma$ is then realized as a quotient of $M_t$. 

Due to the real analytic dependence of the Plateau solution $f_t$ on the angle $t$, this properties extends to 
 irrational angles $t\in(0, \tfrac{1}{4}]\setminus \mathbb Q.$ 
 
 It remains to show that the conformal type of $\underline \Sigma$, i.e, the cross ratio of the four punctures  is $-1$. This follows by using the additional symmetries of the Plateau solution inherited from the symmetries of its geodesic boundary
 $\Gamma_t,$ see \cite[Section 3.3]{He1}.
 \end{proof}
 
 Instead of $\underline \Sigma$ we will be  working on a 2-fold covering  $ \pi\colon \Sigma\to\underline\Sigma$
 branched over $0,\infty$ which removes the conical singularities of the induced metric at $Q_1$ and $Q_2.$ For $t = \tfrac{1}{2(g+1)}$ this gives one handle of the whole genus $g$ minimal surface $\xi_{1,g}$ (given by the compact surface modulo $\Z_{g+1}$ symmetry). The surface $\Sigma$ is a 2-sphere with $4$ marked/singular points 
$p_1,p_2$ and $p_3,p_4$  given by the preimages of $P_1$ and $P_2$ under $\pi,$ respectively. 
 \begin{remark}\label{Moebiuschoice}
By Proposition \ref{pro:cftsigma} we can fix the four marked/singular points of $\Sigma$ up to a M\"obius transformation to be
 either
 \begin{equation}\label{uspunctures}
 z_1= -1, \quad z_2= 1, \quad z_3= 0, \quad  z_4 = \infty
 \end{equation}
 
or

 \begin{equation}\label{uspuncturesp}
 p_1= e^{i \pi/4}, \quad p_2= - e^{i \pi/4}, \quad p_3= e^{i \pi3/4}, \quad  p_4 = -e^{i \pi3/4}.
 \end{equation}
We will make use of both normalizations of these marked points, as this facilitates many computations. The ordering of the $z_k$ and $p_k$ is such that the cross-ratio satisfies $$\text{Xratio}(z_1, ..., z_4) = \text{Xratio}(p_1, ..., p_4)  = -1.$$
 \end{remark}
 
 \section{Fuchsian Systems and parabolic structures}\label{Fuchsian}
 
 \subsection{Logarithmic connections on Riemann surfaces}
 Let $M$ be a compact Riemann surface.
 Let ${\mathbf D}\,=\, p_1+\ldots +p_n$ be a divisor, such that 
the points $p _k \in\, M$ are 
pairwise distinct. 
Let $V\to M$ be a holomorphic vector bundle with underlying holomorphic structure ${\bar \partial}_V$ and let $\mathcal V$ denote the sheaf of its holomorphic sections, i.e., for every open set $U\subset M$
we have
\[H^0(U,\mathcal V)=\{s\in\Gamma(U,V)\mid {\bar \partial}_Vs=0\}.\]

In the following, we always assume $V$ to be a  $\rm{SL}(2, \C)$--bundle, i.e., $V$ is of rank 2 and its determinant line bundle $\Lambda^2V$ is holomorphically trivial.
A \textit{logarithmic} $\rm{SL}(2, \C)$--connection 
$\nabla\,=\,\bar{\partial}_V+\partial^\nabla$ on $V$ with singular part contained in the 
divisor $\mathbf D$ is a holomorphic differential operator
\begin{equation}\label{partialnabla}
\partial^\nabla\, :\, \mathcal V\, \longrightarrow\, \mathcal V\otimes K_M\otimes {\mathcal O}_M({\mathbf D})
\end{equation}
such that
\begin{itemize}
\item the Leibniz rule $\partial^\nabla(fs)\,=\, f\partial^\nabla(s)+ s\otimes 
\partial{f}$ holds for all $s \in \mathcal V$ and $f \in {\mathcal O}_M$, and
\item the induced holomorphic connection on $\Lambda^2 V\,=\, {\mathcal O}_M$ 
is trivial.
\end{itemize}

Note that the connection $\nabla$ 
is automatically flat over $M\setminus\text{supp}(\mathbf D)$.
Conversely, every connection $\nabla$ on $V\rightarrow M\setminus\text{supp}(\mathbf D)$ with
\[\bar\partial^\nabla={\bar \partial}_V\]
such that its connection 1-form $\omega$ with respect to a (local) holomorphic frame of  $V\to U\subset M$
is meromorphic with at most first order poles at $U\cap \text{supp}(\mathbf D)$ is a logarithmic connection. Moreover, the associated residue 
 \[\text{Res}_{p_j}(\nabla)\,\in\,\text{End}(V_{p_j})\] is tracefree at every point $p _j$
in the singular divisor ${\mathbf D}$. Let
$\pm\rho_{k}$ denote the eigenvalues of $\text{Res}_{p_k}(\nabla)$ which are also called the weights of the logarithmic connection. 
The
logarithmic connection $\nabla$ is called {\em non-resonant} if $2\rho_{k}\,\notin\, \mathbb Z$
for all $j=1,\cdots, n$. 
In this case
its local monodromy around $p_k$ is 
conjugate to the diagonal matrix with entries $\exp(\pm 2\pi i\rho_{k})$.

\subsection{Parabolic structures}
A {\it parabolic structure } $\mathcal P$ on a holomorphic $\rm{SL}(2, \C)$--bundle $V$
over the divisor ${\mathbf D}$ is defined by a collection of complex  lines $L_k\, \subset\, V_{p_k}$
together with parabolic weights $\rho_k\, \in\, (0,\, \tfrac{1}{2})$ of $L_k$ for all
$k=1,\cdots, n$.
The corresponding divisor ${\mathbf D}$ is called the parabolic
divisor and $\{L_k\}_{k=1}^n$ are called the quasiparabolic lines.
The parabolic degree of  a holomorphic line subbundle $W\, \subset\, V$  is defined to be
\[\text{par-deg}(W)\,:=\, {\rm deg}(W)+\sum_{k=1}^n \rho^W_k\, ,\]
where $\rho^W_k\,= \, \rho_k$ if $W_{p_k}\,=\,L_k$, and $
\rho^W_k\,=\,
-\rho_k$ if $W_{p_k}\,\neq\, L_k$. The parabolic degree of $V$ is always 0 in our setup.

\begin{remark}
There are different possible definitions for parabolic structures, compare \cite{MS,Sim0,KW,Men}. 
There are also different conventions for the range of the weights.
We use the `trace-free' convention, see \cite{Pir,HH_abel}.
\end{remark}

\begin{definition}[\cite{MS},\cite{Pir}]
A parabolic structure $\mathcal P$
on the $\rm{SL}(2, \C)$--bundle $V$ is called {\it stable} (respectively, {\it semistable}) if
$\text{par-deg}(W)\, <\, 0$ (respectively, $\text{par-deg}(W)\, \leq\, 0$)
for every holomorphic line subbundle $W\,\subset \, V$. A semistable parabolic bundle that is not stable is called \textit{strictly semistable}.
A parabolic bundle which is not semistable is called \textit{unstable}.
\end{definition}

Every non-resonant logarithmic $\rm{SL}(2, \C)$--connection $\nabla$ on a holomorphic bundle $V$ such that the eigenvalues of all residues are contained in the interval
$(-\tfrac{1}{2},\, \tfrac{1}{2})$ naturally induces a parabolic 
structure $\mathcal P$ on $V$. The parabolic divisor of $\mathcal P$ is
hereby the singular locus $\mathbf D\,=\, p_1 
+\ldots + p_n$ of $\nabla$, the parabolic weight
at $p _k$ is the positive eigenvalue $\rho_k$ of $\text{Res}_{p_k}(\nabla)$ and
the quasiparabolic line at $p_k$ is the eigenline for $\text{Res}_{p_k}(\nabla)$ with respect to the eigenvalue $\rho_k$.
 
A {\it strongly parabolic Higgs field} on a parabolic
$\rm{SL}(2, \C)$--bundle $(V,\,{\mathcal P})$ is a meromorphic endomorphism-valued 1-form 
$$
\Psi \, \in\, H^0(M,\, \text{End}(V)\otimes K_{M}\otimes {\mathcal O}_{M}({\mathbf D}))
$$
such that $\text{tr}(\Psi)\,=\,0$  and
such that the residues $\text{Res}_{p_k}\Psi$ are nilpotent with kernels
given by the quasiparabolic lines
$$L _k\, \subset\, {\rm kernel}(\text{Res}_{p_k}\Psi)$$ for all $k = 1,..., n$. 
Note that $\Psi$ has at most first order poles at the singular points $p_k$.

Two  logarithmic $\rm{SL}(2, \C)$--connections $\nabla_1$ and $\nabla_2$
on $V$ with singular part contained in ${\mathbf D}\,=\, p_1+\ldots +p_n$ induce the same parabolic structure on $V$ if and only if
$\nabla_1-\nabla_2$ is a strongly parabolic Higgs field for the parabolic
structure induced by $\nabla_1$ (or equivalently, for the parabolic
structure induced by $\nabla_2$).

A general result of Mehta and Seshadri \cite[p.~226, Theorem 4.1(2)]{MS}, and Biquard 
\cite[p.~246, Th\'eor\`eme 2.5]{Biq} (see also \cite[Theorem 3.2.2]{Pir}) implies that the 
above construction of associating a parabolic bundle to a logarithmic connection
actually gives to a bijection between the space 
of isomorphism classes of irreducible flat ${\rm SU}(2)$--connections on $M \setminus 
{\mathbf D}$ and the stable parabolic $\rm{SL}(2, \C)$--bundles on $(M,\, {\mathbf D})$. 
As a 
consequence, every logarithmic connection $\nabla$ on $V$ giving rise to a stable parabolic 
$\rm{SL}(2, \C)$--structure $\mathcal P$ admits a unique strongly parabolic Higgs field
$\Psi$ on $(V,\, {\mathcal P})$ such that the monodromy representation of  $\nabla+\Psi$ is unitary.

In this context, reducible unitary logarithmic connections induce strictly semi-stable parabolic structures, as a parallel subbundle
$W$ has automatically parabolic degree 0. If the underlying parabolic structure of a logarithmic connection $\nabla$ is unstable
then $\nabla$ is not unitarizable.

 \subsection{Fuchsian systems on the 4-punctured sphere}
A particular class of logarithmic connections is provided by Fuchsian systems. Let $p_1,\dots,p_4\subset \C \subset \C P^1$. A $\SL(2, \C)$ Fuchsian system on the 4-punctured sphere 
is a holomorphic connection on the trivial $\C^2$-bundle over $\C P^1\setminus\{p_1,\dots,p_4\}$ of the form $\nabla= d+ \xi$ with 

\begin{equation}\label{no_residue_at_infty}
\xi = \sum_{k=1}^4 A_k \frac{dz}{z-p_k} \quad \text{such that} \quad \sum_{k=1}^4 A_k=0,
\end{equation}
to avoid a further singularity at $z =\infty.$ If $p_4 = \infty$ then $$\xi = \sum_{k=1}^3 A_k \frac{dz}{z-p_k} \quad \text{such that} \quad A_4 = -\sum_{k=1}^3 A_k.$$
As before, we assume that the residues $A_k \in \mathfrak {sl}(2, \C)$ 
have eigenvalues $\pm \rho_k$ with $\rho_k\in(0,\tfrac{1}{2}),$
i.e., $\nabla$ is non-resonant and
 the conjugacy class of the local monodromy along a simple closed  curve around a puncture is determined by $A_k$ and lies in  the conjugacy class of
 \[\begin{pmatrix}e^{2\pi i \rho_k}&0\\0&e^{-2\pi i \rho_k}\end{pmatrix}.\]
\begin{remark}\label{extension}
 In particular, there is a local holomorphic frame of $V$ such that the connection 1-form of the logarithmic connection is diagonal,
 see \cite{De}.\end{remark}

Two Fuchsian systems are equivalent (when fixing the punctures $p_k$), if there exist an invertible matrix $G$ such that \[\tilde A_k = G^{-1}A_k G\] for $k=1,\dots,4.$ Note that two Fuchsian systems with the same {weights
are equivalent if and only if the connections are gauge equivalent.

 \begin{definition}
 A \SL$(2, \C)$-Fuchsian system $\nabla$ is called reducible, if there exist a $\nabla$-invariant holomorphic line subbundle.  Otherwise, the Fuchsian system is called irreducible.
 \end{definition}
\begin{remark}
In our situation of the 4-punctured sphere we have: if $0<\sum_{k=1}^4\rho_k<1,$ then
being reducible is equivalent to the residues $A_k$, $k=1,...,4$ being simultaneously diagonalizable. Given two irreducible and equivalent Fuchsian systems $\nabla$ and $\wt \nabla$, the SL$(2, \C)$-gauge matrix $G$ between 
them is uniquely determined up to sign. 
\end{remark}
Fuchsian systems admitting a unitary monodromy representation are of particular interest for the construction of minimal surfaces in the 3-sphere.   
 \begin{definition}
 A \SL$(2, \C)$-Fuchsian system is called unitarizable if there exist a hermitian metric $h$  on $\underline{\C}^2\to \Sigma$ such that the connection $d+\xi$ is unitary with respect to $h.$
 \end{definition}

\subsection*{Convention:}In the following we will only consider $\mathrm{SL}(2,\C)$ Fuchsian systems on the 4-punctured sphere $\Sigma$
such that all eigenvalues are $\pm\rho$ with $\rho\in(0,\tfrac{1}{2})$. We call them Fuchsian systems with eigenvalue/weight  $\rho$ for short. We call a Fuchsian system stable or semi-stable depending on the stability of the induced parabolic structure.

\subsection{The parabolic modulus and coordinates}\label{sec:modulus}
Before defining coordinates for the moduli space of Fuchsian systems, we first collect and prove some folklore facts about these. Though we believe these facts to be well-known to experts and can easily be deduced from \cite{LS}, we include the proofs here to make the paper more self-contained. In this section we will normalize the punctures $z_1, ..., z_4$ of $\Sigma$ as in \eqref{uspunctures},
i.e., $z_1=-1,z_2=0,z_3=1,z_4=\infty.$

\begin{lemma}\label{lemmastable}
Let $\nabla$ be a \SL$(2,\C)$-Fuchsian system on the 4-punctured sphere  with parabolic weights $\pm\rho$ and $\rho \in (0, \tfrac{1}{4})$. Then the induced parabolic structure is semi-stable.  Moreover, the induced parabolic structure is strictly semi-stable if and only if two of the four quasiparabolic lines coalesce.

\end{lemma}

\begin{proof}
We have to show that every holomorphic line subbundle $W$ of $V = \mathcal O \oplus \mathcal O$ has non-positive parabolic degree. Since $W$ is holomorphic, deg $W \leq 0$ and moreover, if deg $W< 0$ then also its parabolic degree is negative, as $4\rho<1$ by assumption. 

A holomorphic subbundle $W$ of degree 0 of $\mathcal O\oplus\mathcal O$ is {\em constant}, i.e., parallel with respect to the trivial connection $d$. Write $\nabla$ with respect to $V= W \oplus \wt W$ for a complementary holomorphic line bundle $\wt W$ of $W$ in $V.$  Then the lower-left entry $\beta^W$ of $\nabla$ is a meromorphic $1$-form with at most first order poles at the four singular points. Since $L_k$ is a eigenline of the residue of $\nabla$ at $p_k$,  we obtain Res$_{p_k}\beta^W = 0$,  whenever $W = L_k$ at $p_k.$ In this case $\beta^W$ is holomorphic at $p_k. $ 

Therefore, if more than two quasiparabolic lines coalesce, $\beta^W \equiv 0,$ as there  are no non-zero meromorphic $1$-forms with at most one pole of order 1 on $\C P^1.$ But this would imply that $W$ is a parallel line subbundle of $V.$ Then the generalized residue
formula
gives
\[0 = \deg(W)+\sum_{k=1}^4 \rho^W_k \geq 2 \rho\]
contradicting $\rho>0.$
Therefore, 
$$\text{par-deg } W \leq -2\rho + 2 \rho = 0,$$
with equality if and only if two of the quasiparabolic lines $L_1,\dots,L_4$ coalesce and $W$ coincides with these two lines.
\end{proof}

The next Lemma characterises all reducible and unitary Fuchsian systems. 

\begin{lemma}\label{reducible}
Let $\nabla$ be a reducible and unitarizable Fuchsian system on the 4-punctured sphere  with parabolic weights $\pm\rho$ and $\rho \in (0, \tfrac{1}{4})$, then up to conjugation $\nabla$ is given by 
\[d+\sigma_{-1}\begin{pmatrix} \rho&0\\0&-\rho\end{pmatrix} \frac{dz}{z+1}+\begin{pmatrix} \rho&0\\0&-\rho\end{pmatrix} \frac{dz}{z}+\sigma_1 \begin{pmatrix} \rho&0\\0&-\rho\end{pmatrix} \frac{dz}{z-1},
\] where $(\sigma_{-1},\sigma_1)\in\{(-1,-1),(-1,1),(1,-1)\}.$ 
\end{lemma}

\begin{proof}\
A reducible Fuchsian system $\nabla$ possess an invariant holomorphic line subbundle $W$.
As before the generalized residue formula then gives that 
\[0 = \deg(W)+\sum_{k=1}^4 \rho^W_k = \text{par-deg }W.\]
Since $\rho \in (0,\tfrac{1}{4}),$ this implies that the degree of $W$ is 0 and the induced parabolic weights $\rho_k^W$ on $W$ sum up to $0.$ 
Since $\nabla$ is unitarizable there exists hermitian metric $h$ for which $\nabla$ is unitary. 
Let $W^\perp$ denote the orthogonal complement of $W$ $h$, then $W^\perp$ is parallel as well. Moreover, by Remark \ref{extension}, the holomorphic structure of $W \oplus W^\perp$ extends through the punctures as $\nabla$ is non-resonant. 
Therefore, $\nabla$ is of the stated form with respect to the splitting $V= W \oplus W^\perp$.
\end{proof}
Consider on $\Sigma$ the involutions 
$$\delta(z)=-\frac{1}{z}\quad \text{ and } \quad \tau(z)=\frac{1-z}{z+1}$$  which interchanges the punctures $z_1, ..., z_4$.  We call a Fuchsian system symmetric if there exist $\tilde D,\tilde C\in\mathrm{SL}(2,\C)$ such that 
$$\delta^*\nabla=\nabla.\tilde D\quad \text{ and } \quad \tau^*\nabla=\nabla.\tilde C.$$
\begin{remark}
The symmetries $\delta,\tau$ differ from those in \cite{HHT2} only by the Moebius transformation interchanging the points $z_k$ \eqref{uspunctures} and $p_k$ \eqref{uspuncturesp}. Moreover, we will show in Lemma \ref{lemma_modulipara} that up to conjugation the matrices $\tilde C$ and $\tilde D$ can be chosen to be 
\begin{equation}\label{CD}
D = \begin{pmatrix} i & 0 \\ 0& -i\end{pmatrix} \quad \text{and} \quad C = \begin{pmatrix}0 & i \\ i &0\end{pmatrix}.
\end{equation}
\end{remark}

We can characterize symmetric and reducible Fuchsian systems and strictly semi-stable parabolic structures.
\begin{lemma}\label{reduciblesymmetric}
Let $\nabla$ be a symmetric SL$(2,\C)$-Fuchsian system on the 4-punctured sphere $\Sigma$. If $\nabla$ is strictly semi-stable, then its parabolic structure $\mathcal P$ coincides with parabolic structure of a reducible and unitary Fuchsian system.  Moreover, if $\nabla$ is reducible, then it is automatically 
unitarizable.
\end{lemma}

\begin{proof}
Let $\nabla$ be a Fuchsian system that induces a strictly semi-stable parabolic structure $\mathcal P.$ By Lemma \ref{lemmastable} there exists
 a constant line subbundle $W$  of $V$ with parabolic degree 0 which coincides with two of the four quasiparabolic lines. Denote the corresponding singular points by $a\neq b\in\{z_1,\dots,z_4\}$, and the remaining two singular points by $c\neq d\in\{z_1,\dots,z_4\}\setminus\{a,b\}$.

Consider the group of automorphisms generated by the involutions $\delta$ and $\tau$ (this group is isomorphic to $\Z_2\times \Z_2$).  
In this group there exist a unique element $\mu$ interchanging the two pairs $\{a,b\}$ and $\{c,d\}$ of singular points. Since $\nabla$ is symmetric, we have $\mu^*\nabla=\nabla. M$ for some $M\in\mathrm{SL}(2,\C).$ Moreover,
\[MW\neq W,\]
otherwise $W$ would coincide with all four quasiparabolic lines contradicting Lemma \ref{lemmastable}.  Hence, 
$$V= \mathcal O\oplus\mathcal O=W\oplus MW,$$ such that the quasiparabolic lines at $a,b$ are $W$ and the quasiparabolic lines
at $c,d$ are $MW.$ In particular, since the Fuchsian system is non-resonant, 
its parabolic structure
is (gauge equivalent to) the one obtained from a reducible unitary Fuchsian system stated in Lemma \ref{reducible}. Therefore, $\nabla$ differs from one of the 3 connections of Lemma \ref{reducible} by a strongly parabolic Higgs field $\Psi$ which is necessarily off-diagonal with respect to $W \oplus MW$. Since all reducible connections 
on this parabolic bundle
must be diagonal with respect to $W \oplus MW$, the strongly parabolic Higgs field $\Psi$ must be zero  implying that $\nabla$ is unitary. 

\end{proof}

Two Fuchsian systems inducing the same parabolic structure differ by a strong parabolic Higgs field. The next Lemma shows that the space of these Higgs fields is complex 1 dimensional. For strictly semi-stable parabolic structures the space is actually bigger, therefore we have to restrict to symmetric Higgs fields here.
\begin{lemma} \label{Higgsfield}
Let $\nabla$ be a stable Fuchsian system. Then the space $\mathcal H$ of strongly parabolic Higgs fields is a complex 1-dimensional vector space. If $\nabla$ is symmetric and strictly semi-stable, the space of symmetric strongly parabolic Higgs fields is complex 1-dimensional.
\end{lemma}

\begin{proof}
First, consider a strongly parabolic Higgs field with respect to a stable parabolic structure $\mathcal P$. Assume that $\mathcal H$ is at least 2-dimensional, i.e., there exist linear independent strongly parabolic Higgs fields $\Psi_1$ and $\Psi_2$. Recall that the residues of strongly parabolic Higgs fields are nilpotent with kernel given by the respective
quasiparabolic line. 
Therefore, we can take a non-zero linear combination of $\Psi_1$ and $\Psi_2$ such that such that  the residue at $z_1$ vanishes. Up to conjugation and scaling, the residues at
$z_2,$ $z_3$ and $z_4$ are 
\[\begin{pmatrix}0&1\\0&0\end{pmatrix},  \quad \begin{pmatrix}x&y\\z&-x\end{pmatrix}, \quad \text{and} \quad \begin{pmatrix}-x&-y-1\\-z&x\end{pmatrix} , \quad \text{respectively.} \]
As the residues at $z_3$ and at $z_4$ are both nilpotent, we get $-x^2-yz=0$ and $-x^2-(y+1)z=0$ which implies $z=0$ and $x=0$. This gives that the kernel of the residue at $z_3$ and $z_3$ is the same, which contradicts the fact that the quasiparabolic lines are pairwise distinct for a stable parabolic structure. A non-trivial parabolic Higgs field for every stable parabolic structure is given in \eqref{psi^u} below.

By Lemma \ref{reduciblesymmetric} the parabolic structure of a symmetric, and strictly semi-stable Fuchsian system $\nabla$ is the one of a reducible unitary Fuchsian system, i.e, by Lemma \ref{reducible} it is diagonal with respect to the splitting $V= W\oplus MW$ with  off-diagonal strongly parabolic Higgs field. 
Moreover, the upper right entry  is $x (\frac{dz}{z-a}-\frac{dz}{z-b})$  for some $x\in\C$ by residue theorem and the lower left entry is
$x\mu^*(\frac{dz}{z-a}-\frac{dz}{z-b})$ by symmetry.
\end{proof}

By Lemma \ref{lemmastable}
 a stable Fuchsian system $\nabla$ induces a stable parabolic structure $\mathcal P$ with four pairwise distinct quasiparabolic lines, i.e., 
the four eigenlines with respect to the positive eigenvalues of the four residues are pairwise distinct lines in $\C^2.$ Hence, their cross-ratio, denoted by $u$, is well-defined after some choices. Without loss of generality we can conjugate $\nabla$ by a suitable $g\in\mathrm{SL}(2,\C)$ such that
the eigenlines with respect to the positive weight
 at $z_2=0$, $z_3=1$ and $z_4=\infty$
are
\[\C\begin{pmatrix}0\\1\end{pmatrix}, \quad \quad \C\begin{pmatrix}1\\1\end{pmatrix},  \quad \text{and }
\quad \C\begin{pmatrix}1\\0\end{pmatrix}, \quad \text{respectively}.\]
Then, there exists a unique 
\[u:=u(\nabla)\in \C P^1\setminus\{0,1,\infty\}\] such that the eigenline of  the residue of $\nabla$ at $z_1=-1$ is
\[\C\begin{pmatrix}u\\1\end{pmatrix}.\]

Note that the gauge $g$ is unique up to sign, and hence we have fixed the gauge freedom.
In particular, if $\nabla$ and $\widehat\nabla$ have different $u(\nabla)\neq u(\widehat\nabla)$ then
$\nabla$ and $\widehat\nabla$ cannot be gauge-equivalent. This gives rise to a well-defined holomorphic map 
\begin{equation}\label{u1}u\colon \{\text{Fuchsian systems with stable parabolic structure}\}\to\C P^1,
\end{equation}
which we refer to as {\em modulus map}. This map is invariant under conjugation (or gauge transformation) by construction.
In particular, a holomorphic family of Fuchsian systems $\lambda\in U\subset\C\mapsto \xi(\lambda)$ 
with stable underlying parabolic structures
gives rise to a holomorphic function $u\colon U\to\C.$
It is not difficult to prove (using Riemann's theorem about removable singularities) that
$u$ extends  to the space of all Fuchsian systems. More precisely, every holomorphic
family of Fuchsian systems $\lambda\in U\subset\C\mapsto \xi(\lambda)$ 
gives rise to a holomorphic modulus function $u\colon U\to\C.$
In particular,  the images of the reducible Fuchsian systems 
of Lemma \ref{reducible} 
(parametrised by $(\sigma_{-1},\sigma_1)\in \{(-1,-1),(-1,1),(1,-1)\}$)  under $u$
are given by
\[(-1,-1)\mapsto_u 1;\quad (-1,1)\mapsto_u \infty; \quad  (1,-1)\mapsto_u 0.\]

Now we can introduce global coordinates  
of the moduli space of stable, trace-free Fuchsian systems on a $4$-punctured sphere.  To our best knowledge these have been first  been introduced by Loray-Saito \cite{LS}, see also \cite{HH_abel}. Let $\rho \in (0, \tfrac{1}{2})$.
For $u\in \C\setminus\{0,1\}$ set
\begin{equation}\label{concrete_parabolic_structure}
\begin{split}
A^u_1=\begin{pmatrix}-\rho  &2\rho u \\ 0 & \rho\end{pmatrix}, \quad\,A^u_2=\begin{pmatrix}-\rho  &0 \\ -2\rho & \rho\end{pmatrix}, \quad 
A^u_3&=\begin{pmatrix}\rho &0  \\ 2\rho & - \rho\end{pmatrix}, \quad
A^u_4=\begin{pmatrix}\rho &-2\rho u \\ 0 & -\rho\end{pmatrix}.
\end{split}
\end{equation}

Then the connection 

\[\nabla^u:=d+\sum_{k=1}^3 A_k^u \frac{dz}{z-z_k}\] is a Fuchsian system with poles at $z_k$ for $k= 1, ..., 4$ and parabolic weights $\pm\rho$ and modulus $u.$   Moreover, for
\begin{equation}\label{concrete_Higgs}
\begin{split}
\Psi_1=\begin{pmatrix}-u &u^2 \\ -1 & u\end{pmatrix}, \quad\Psi_2=\begin{pmatrix}0 &0 \\ 1-u & 0\end{pmatrix},\quad
\Psi_3=\begin{pmatrix}u&-u \\ u & -u\end{pmatrix},\quad
\Psi_4=\begin{pmatrix}0 &u-u^2 \\ 0 & 0\end{pmatrix},
\end{split}
\end{equation}
the $1$-form
\begin{equation}\label{psi^u}
\Psi^u=\Psi:=\sum_{k=1}^4 \Psi_k \frac{dz}{z-p_k}
\end{equation}
 is a strongly parabolic Higgs field with respect to the modulus $u.$  By Lemma \ref{Higgsfield} the Higgsfield $\Psi^u$ is unique up to scaling for $u \in \C \setminus\{0,1\}$. Therefore, we can write every stable Fuchsian system up to a unique gauge as 
\begin{equation}\label{nablaus}\nabla^{u,s}:=\nabla^u+ s \Psi, \quad s\in \C\end{equation}
for some unique $s\in\C$.  This gives us a global coordinate system $(u,s)$ on the moduli space of stable Fuchsian systems. We remark that $u=-1$ is special in the sense that it is the unique {\em stable} parabolic structure
which admits a nilpotent strongly parabolic Higgs field, as 
\[\text{det}(\Psi^u)=-\frac{u-u^3}{z-z^3}.\]

The other exceptional cases $u = 0, 1, \infty$ correspond to strictly semi-stable parabolic structures which are not symmetric with respect to the involutions $\delta$ and $\tau$, since they have three distinct quasiparabolic lines.

\subsection{Exceptional logarithmic connections}
Every representation of the first fundamental group of the 4-punctured sphere
with prescribed local monodromies is induced by a logarithmic connection. But not all of the logarithmic connections are Fuchsian, i.e., they are not always defined on the trivial holomorphic bundle $\mathcal O\oplus\mathcal O.$

By Grothendieck splitting the underlying holomorphic vector bundle $V$ of a logarithmic connection over $\C P^1$ is given by $V= \mathcal O(l) \oplus \mathcal O(-l)$ for $l \in \N$. On the $4$-punctured sphere we have $l \leq 1$ due to degree considerations. In fact, for $l>1$, the second fundamental form of $\mathcal O(l)$ necessarily vanishes, i.e., $\mathcal O(l)$ would be parallel, 
contradicting the generalized residue formula with $\rho \in (0, \tfrac{1}{4})$. We have already studied the case of $V= \mathcal O \oplus \mathcal O$, where the parabolic structure is always semi-stable, and the complex dimension is $2$ of the moduli space of Fuchsian systems with prescribed weights is 2. In contrast logarithmic connections on $V= \mathcal O(-1) \oplus \mathcal =(1)$ turn out to be unstable and the moduli space is complex 1-dimensional.

\begin{proposition}\label{up1extension}
Let $\rho\in(0,\tfrac{1}{4})$ 
and $V= \mathcal O(-1) \oplus \mathcal O(1)\rightarrow \Sigma$. Then, the moduli space of trace free logarithmic connections 
on $V$ with four singular points with parabolic weight $\rho\in(0, \tfrac{1}{4})$ is a complex line. 
The underlying parabolic structure is unstable.

Moreover, the parabolic modulus $u$ as a holomorphic function of the moduli space of logarithmic connections
to $\mathbb C P^1$ 
extends holomorphically to this complex line, mapping the whole line to $u=-1.$ 
\end{proposition}
\begin{remark}
The moduli space of parabolic structures admitting logarithmic connections is therefore not Hausdorff, but has a double
point over $u=-1.$
\end{remark}
\begin{proof}
The parabolic structure is unstable, as the line bundle $\mathcal O(1)$ is a holomorphic subbundle of $V$ with positive parabolic degree on the $4$-puncture sphere as $\rho \in (0, \tfrac{1}{4}).$

To determine the moduli space of logarithmic connections on $V$, we decompose a  trace-free logarithmic connection $\nabla$ with respect to  $V= \mathcal O(-1) \oplus \mathcal O(1)$
as
\[\nabla=\begin{pmatrix} \nabla^{-1}&\alpha\\\beta&\nabla^{1}\end{pmatrix}\]
where $\nabla^1$ and  $\nabla^{-1}$ are the induced logarithmic connections on $\mathcal O(1)$ and $\mathcal O(-1)$, respectively, 
$\alpha$ is a meromorphic $\mathcal O(-2)$-valued 1-form with at most 4 simple poles and $\beta$ is a meromorphic $\mathcal O(2)$-valued 1-form with at most 4 simple poles.
Hence $\alpha$ has exactly 4 poles, i.e., up to scaling it is the unique  meromorphic  section  of $\mathcal O(-4)=\mathcal O(-2)\otimes K_{\mathbb CP^1}$ with simple poles at $z_1,\dots,z_4.$ Note that
$\alpha$ is non-zero since otherwise $\mathcal O(1)$ would be parallel contradicting the generalized residue formula
\[1=\text{deg}(\mathcal O(1))=-\text{Res}(\nabla^1)=-\sum_{k=1}^4 \pm \rho\neq1\]
for $\rho\in(0,\tfrac{1}{4}).$
Any holomorphic $\mathrm{SL}(2,\C)$ gauge transformations of
$V$ are of the form
\[g=\begin{pmatrix} a & 0\\P&\tfrac{1}{a}\end{pmatrix},\]
where $a\in\C^*$ and $P\in\mathcal H^0(\C P^1,O(2)).$ To fix $P$ we can choose its values at 3 of the four singular points, e.g., at $z_1,\dots,z_3$, 
such that
\[\text{Res}_{z_k}(\nabla^{-1}+a^{-1}\alpha P)=\rho\]
for
$k=1,\dots,3.$
The gauged connection is then given by 
\[\nabla.g=\begin{pmatrix} \nabla^{-1}+a^{-1}\alpha P&a^{-2}\alpha\\ a^2\beta-a\nabla P-\alpha P^2&\nabla^{1}-a^{-1}\alpha P\end{pmatrix}.\]
By construction the upper left diagonal entry has first order poles at  $z_1, ..., z_4$ with residue $\rho$ at $z_1, ..., z_3.$ The upper right entry  $a^{-2}\alpha$ has poles of order 1 at $z_1,\dots,z_4.$ Therefore,
the lower left entry cannot have poles at $z_1,\dots,z_3$ since the eigenvalues of the residues of $\nabla .g$ are all $\pm \rho$. 
Moreover, the lower left entry must have a pole at $z_4$.
Thus, after a suitable choice of $a\in \C^*$ we can assume that the lower left entry of $\nabla.g$ is 
of the form
\[(c_0+z)dz,\]
for some $c_0\in\C$,
with respect to
the standard trivialisation of the bundle 
$\mathcal O(-1)\oplus\mathcal O(1)$ over $\C$ (given by sections with first order pole or zero, respectively,  at $z=z_4=\infty$). To be more explicit, let
\[h=\begin{pmatrix}z&0\\0&z^{-1}\end{pmatrix}\]
be the cocyle for the bundle $V$, i.e.,  the standard frame of $V$ is given by $V_0= \underline{\C}^2\rightarrow U_0\subset \C P^1$ and $V_\infty=\underline{\C}^2\rightarrow U_\infty\subset \C P^1$ with transition function $h$ on $U_0 \cap U_\infty$ mapping $V_0$ to $V_\infty$.
Since the eigenvalues of the residue of $\nabla$ at $z=\infty$ are $\pm \rho,$ we obtain that the connection 1-form with respect to the standard frame on $U_0$
is given by
\begin{equation}\label{connection1formU0}\omega=\begin{pmatrix}
\frac{\rho(1-3 z^2)}{z-z^3}&\frac{8 \rho^2-6 \rho+1}{z-z^3}\\ c_0+z&-\frac{\rho(1-3z^2)}{z-z^3}. 
\end{pmatrix}dz.\end{equation}
Note that the connection 1-form for the frame over $U_\infty$ is given by $\omega=hdh^{-1}+h\omega h^{-1}$,
which has a first order pole at $z=\infty.$
The affine line of logarithmic connections on $V$ is thus parametrized by $c_0\in\C.$

For the second part of the Proposition, note that the moduli space  of logarithmic connections with prescribed local monodromies 
is a smooth complex manifold of dimension 2  away from reducible connections. Therefore, we introduce a second parameter $E$ and consider the holomorphic rank $2$ bundle $V(E)$ over $\C P^1$ defined by
the cocycle 
\[h^E=\begin{pmatrix} z &- E\\0&z^{-1}\end{pmatrix}.\]

on $\C^*=U_0 \cap U_\infty.$ For $E\neq0$ and with respect to the standard frame over $U_0$ the sections $s_1=\begin{pmatrix}1\\\tfrac{z}{E}\end{pmatrix}$ and  $s_2=\begin{pmatrix}0\\1\end{pmatrix}$ extends to global holomorphic sections
without zeros of determinant $s_1\wedge s_2=1.$ Therefore, the corresponding holomorphic bundle is trivial, 
while for $E=0$ we have $V(E)= \mathcal O(-1)\oplus\mathcal O(1)$.

Let $\nabla^{E, c_0}$ be the logarithmic connection on $V(E)$ with connection 1-form over $U_0$
\[\omega^{E,c_0}=\left(
\begin{array}{cc}
 \frac{\rho(1-3 z^2)}{z-z^3}+E & \frac{-c_0^2 E^2+6 \rho (c_0 E-E z+1)+c_0 E (E z-2)-8 \rho^2-E^2 z^2+E^2+E z-1}{z^3-z} \\
 c_0+z & -\frac{\rho(1-3 z^2)-E z^3+E z}{z-z^3} \\
\end{array}
\right),\]
which coincides with \eqref{connection1formU0} for $E=0$, and parametrizes an open neighborhood of the unstable line inside the moduli space of irreducible logarithmic connections.
Consider the $z$-independent matrix
\[C=\left(
\begin{array}{cc}
 \frac{c_0 E-2 \rho+1}{E (c_0 E-2 \rho+E+1)} & 0 \\
 -\frac{E}{c_0 E-2 \rho+E+1} & 1 \\
\end{array}
\right)\]

and the gauge
$$l^E=\left(
\begin{array}{cc}
 0 & 1 \\
 -1 & \frac{z}{E} \\
\end{array}
\right).$$

Then  direct computation gives
\[(\nabla).(l^E C)= \nabla^{u_0,s_0}\]
where $\nabla^{u_0,s_0}$ is the Fuchsian system defined in \eqref{nablaus} with parameters given by 
\begin{equation}
\begin{split}
u_0&=-\frac{(c_0+1) E+1-2\rho}{(c_0-1) E+1-2 \rho}\\
s_0 &=-\frac{( (1-c_0) E+2 \rho-1) ((1-c_0) E+4 \rho-1)}{2 E}.
\end{split}
\end{equation}
In particular,  the function $u=u(E,c_0)$ is holomorphic in a neighbourhood of $E=0$, and maps 
the  exceptional line $\{E=0,c_0\in\C\}$ of logarithmic connections on $\mathcal O(-1)\oplus\mathcal O(1)$ to $u_0=-1.$
\end{proof}

\begin{remark}
A useful observation is that for the family of logarithmic connections $\nabla^{E, c_0}$ defined in the proof, the product $(u+1)\cdot s$ has the simple form
\begin{equation}\label{up1sexpansion}(u(E,c_0)+1)\cdot s(E,c_0)=1-4\rho+(c_0-1) E\end{equation}
which extends holomorphically to $E=0.$
In particular, a unstable logarithmic connection can only appear as the limit of a  family of stable Fuchsian systems $\lambda \mapsto \nabla^{u_\lambda, s_\lambda}$  for $\lambda\to\lambda_0\in\C^*$
if and only if ${\displaystyle \lim_{\lambda \rightarrow \lambda_0}} (u_\lambda+1)s_\lambda = 1-4 \rho.$ \end{remark}

\begin{lemma}\label{lemma_modulipara}
Let $\nabla$ be a SL$(2, \C)$ logarithmic connection on $\C P^1$ with four singular points
$z_1, ..., z_4 $ and
with parabolic weights $\rho \in (0, \tfrac{1}{4})$, Let $\nabla$ be symmetric
with respect to the symmetries $\delta$ and $\tau$.  If $\nabla$ has unitarizable monodromy, then $\nabla$ is a Fuchsian system, i.e, the underlying holomorphic bundle is $V= \mathcal O \oplus \mathcal O$.  

Moreover, if  $\nabla$ is Fuchsian, then its gauge class can be 
 represented (unique up to conjugation with elements of the finite group generated by $D$ and $C$ defined in \eqref{CD}) by
 \[d+\sum_{k=1}^3A_k\frac{dz}{z-z_k,}\]
where $A_k\in\mathfrak{sl}(2,\C)$ satisfying
$A_3=D^{-1}A_1D$, $A_4= -A_1-A_2-A_3=D^{-1}A_2D$.  
\end{lemma}

\begin{proof}
Since a unitarizable logarithmic connection $\nabla$ has semi-stable parabolic structure by the Mehta-Seshadri theorem \cite{MS}, the underlying holomorphic bundle $V$ is trivial as a consequence of Proposition \ref{up1extension}. Therefore, $\nabla$ is a Fuchsian system.

To show that every  Fuchsian system $\nabla$ with parabolic weight $\rho$ satisfying the desired symmetries, we distinguish between irreducible and reducible $\nabla$.

In the reducible case, we can apply Lemma \ref{reducible} to obtain that $\nabla$ is given by
\[d+\sigma_{-1}\begin{pmatrix} \rho&0\\0&-\rho\end{pmatrix} \frac{dz}{z+1}+\begin{pmatrix} \rho&0\\0&-\rho\end{pmatrix} \frac{dz}{z}+\sigma_1 \begin{pmatrix} \rho&0\\0&-\rho\end{pmatrix} \frac{dz}{z-1},
\] where $(\sigma_{-1},\sigma_1)\in\{(-1,-1),(-1,1),(1,-1)\}$ up to conjugation. 
From here it is a straight forward computation to see that $\nabla$ has the desired symmetries after a suitable gauge.

In the second case, $\nabla$ is irreducible and $V$ has no parallel line bundles. Then,  since $\nabla$ is symmetric, there  exist  $\tilde C,\tilde D\in\mathrm{SL}(2,\C)$ such that
\[\delta^*\nabla=\nabla.\tilde D=\tilde D^{-1}\nabla\tilde D\]
and 
\[\tau^*\nabla=\nabla.\tilde C=\tilde C^{-1}\nabla\tilde C.\]
From irreducibility $\tilde C$ and
$\tilde D$ are unique up to sign. Since $\delta^2=\tau^2 = $Id 
 this implies $\tilde D^2=\pm$ Id and $\tilde C^2=\pm$ Id.

 If $\tilde D^2=$Id,  we would have $\tilde D=\pm$Id and therefore writing
\[\nabla=d+ \sum_{k=1}^3 B_k \frac{dz}{z-p_k}\]
yields 
\[B_3=B_1,\quad B_4=B_2 \quad \text{and} \quad B_1+B_2+B_3+B_4=0.\]
Therefore, $\nabla$ would be reducible which is a contradiction.  Thus $\tilde D^2=-$Id, and hence $\tilde D$ is conjugate to $D,$   i.e., there is $g\in\mathrm{SL}(2,\C)$ with
\[\tilde D=gDg^{-1}.\]
The connection $\hat\nabla.g$ then satisfies  $\delta^*\hat\nabla=\hat \nabla.D$ implying the $\delta$-symmetry $\hat\nabla.g$.

For the $\tau$-symmetry  we first use
\[\delta\circ\tau=\tau\circ\delta.\]
Therefore, by irreducibility
\[D\tilde C=\pm\tilde CD.\] 
This implies $\tilde C$ being either diagonal or off-diagonal. If $\tilde C$ is diagonal, $\tilde C^2=\pm$Id yields again reducibility of $\nabla$.
Thus $\tilde C$ must be off-diagonal and $\tilde C^2=-$Id. This gives
\[\tilde C=\pm C\quad\text{or}\quad \tilde C=\pm\begin{pmatrix} 0&1\\-1&0\end{pmatrix} .\]
In the first case $\tilde C$ is of the desired form. In the latter case, $\tilde C$ and $\pm C$ differ by the conjugation with 
\[S^{-1}= \begin{pmatrix} e^{-\frac{\pi i}{4}}&0\\0&e^{\frac{\pi i}{4}}\end{pmatrix}.\]
Uniqueness follows again from irreducibility and the fact that
$\pm$Id are the only $\SL(2,\C)$ matrices commuting with both $C$ and $D$. Moreover, the the space of matrices which either commute or anti-commute 
with $C$ and $D$ are spanned by Id, $C$, $D$ and $CD$. 
\end{proof}

\section{Existences of Fuchsian DPW potentials}\label{sec:existenceofpot}

In this section we want to bring together the analytic properties of the Plateau solutions in Section \ref{Lawson} with the properties of the moduli space of Fuchsian systems in Section \ref{Fuchsian} to show the existence of a Fuchsian DPW potential for the Lawson surfaces $\xi_{1,g}$ for every $g \geq1.$
Since this section is heavily based on \cite{HHT2}, we change to the notations and normalizations of \cite{HHT2} in the following. In particular, we choose the 4-punctures $\Sigma$ to be $p_1, ..., p_k$ as in \eqref{uspuncturesp} and adjust the symmetries $\delta$ and $\tau$ accordingly. These two setups differ only by a M\"obius transformation mapping $z_1,..., z_k$ to $p_1,..., p_k$ and do not affect any properties we have shown. Moreover, the weights of the Fuchsian systems in this section will be $t\in (0, \tfrac{1}{4})$ instead of $\rho.$

\begin{definition}
A DPW potential on a Riemann surface $M$ is a closed (i.e., holomorphic) complex linear 1-form

\[\eta\in\Omega^{1,0}(M,\Lambda \mathfrak{sl}(2,\C))\]

with values in the  loop algebra $$ \Lambda\mathfrak{sl}(2,\C):=\{\xi\colon S^1\to \mathfrak{sl}(2,\C)\mid \xi \text{ is real analytic }\}$$ such that $\lambda \eta$ extends  holomorphically to the entire unit disc $\mathbb D\subset \C.$
Moreover, its residue at $\lambda=0$ 
\[\eta_{-1}:=\text{Res}_{\lambda=0} (\eta)\]
is a nowhere vanishing and nilpotent 1-form. 
\end{definition}

To obtain a well-defined surface, the DPW potential $\eta$ must be unitarizable for $\lambda \in S^1.$ Therefore, its values on the punctured disc $\D_{1+\varepsilon}$ determines the gauge class of $\eta$ for all $\lambda \in \C^*$ via Schwarzian reflexion principle. For $M = \Sigma$ being the $4$-punctured sphere, there exist a particular simple class 
of DPW potentials for which $\eta$ is  a Fuchsian system for all $\lambda \in \overline\D^*_{1+ \varepsilon}\subset\C^*$ for some small $\varepsilon >0.$ We will refer to these potentials as {\em Fuchsian potentials} in the following. 

\subsection{The local structure}
The following results are 
 a generalization of \cite[Proposition 4.2]{BoHeSch} to arbitrary angles $t\in(0,\tfrac{1}{4}]$.  By Theorem \ref{realanalyticity} all Gauss-Codazzi data depend real analytically on $t$.
 
 \begin{lemma}\label{lem:GCdata-dependence}
The first fundamental form of $f_t \colon \Sigma \rightarrow \mathbb S^3$ has a conical singularity with cone angle $4\pi t$ at $p_1, ...p_4$. The Hopf differential $\mathcal Q$ is given by a meromorphic quadratic differential with first order poles at the four singular points. 
\end{lemma} 
\begin{proof}
For rational angles $t\in (0, \tfrac{1}{4}]\cap\mathbb Q,$ the claim follows from the discussion in \cite[Section 4]{Lawson} by going to the quotient $\Sigma$.
For irrational $t$, the claim then follows from the real analytic dependence of the Plateau solutions in $t$, see Theorem \ref{realanalyticity}.
\end{proof}
 
\begin{proposition}\label{pro:localDPW}
For each singular point $p_k$ there exists an open neighbourhood $U_k\subset\Sigma$ and 
 a Fuchsian DPW potential $d+ \xi_t$ for the equivariant minimal surface restricted to the universal covering $\wt U_k\to U_k\setminus\{p_k\}$. The eigenvalues of the residue at $z=p_k$ of the Fuchsian DPW potential
are independent of $\lambda$ and given by $\pm t.$
\end{proposition} 
\begin{proof}
For rational $t$ the Proposition follows from \cite[Section 4]{BoHeSch} using Dorfmeister's normalized potential, see \cite{Wu}. Since the entries of the normalized potential are determined by the Hopf differential and the first fundamental form, the general result is due to the real analyticity of these data in $t$.
\end{proof}

\subsection{The global structure}\label{globalDPW}
Following 
\cite[Section 4]{BoHeSch} there exist two families of flat connections
for the minimal surface $f_t$ restricted to each $U_k$. The first is given by the local DPW potential $d+\xi_t$ provided by Proposition \ref{pro:localDPW}. The second is given by the associated $\C_*$-family of flat connections $\nabla^\lambda_t$ on $\Sigma\setminus\{p_1,\dots,p_4\}$ obtained from the harmonic map $f_t$ into $\mathbb S^3$, see \cite{Hitchin}. By construction of the DPW potential, the two families of flat connections are gauge equivalent on $U_k\setminus\{p_k\},$ $k= 1,..., 4,$ by a positive gauge in $\lambda$. Through these local gauges, we can replace $\nabla^\lambda_t$ on $U_k$  by the Fuchsian potential $d+ \xi_t$ to obtain the following result:
\begin{theorem}\label{1stnabla}
On $\Sigma$ there exist a family of flat connections $\wt\nabla^\lambda_t$ with the following properties:

\begin{itemize}
\item $\wt\nabla^{\lambda}_t$ is a $\C_*$-family of logarithmic SL$(2, \C)$-connections  with singular part $p_1+ ...+p_4$;
\item the parabolic weights of $\wt\nabla^{\lambda}_t$ are $\pm t$;
\item  $\wt\nabla^{\lambda}_t$ are unitarizable for all $\lambda\in \mathbb S^1;$ 
\item $\wt\nabla^{\lambda}_t$  determines the complete minimal surface $f_t$ through the DPW recipe, see \cite[Theorem 1.2]{HHS} or \cite[Section 1.4]{HHT}.
\end{itemize}

\end{theorem}

The theorem shows that for every $\lambda$ fixed the connection $\wt\nabla^\lambda$ is a logarithmic connection on the $4$-punctured sphere $\Sigma.$ By construction of the minimal surfaces, $\Sigma$ and the induced fundamental forms have the intrinsic symmetries
$$\delta(z)=-z,  \quad \tau(z)= \tfrac{1}{z}, \quad  \text{ and }\quad  \sigma(z) = \bar z.$$
The symmetries $\delta$ and $\tau$ are same as the ones considered in Section \ref{Fuchsian} up to the M\"obius
transformation between the punctures.
As shown in \cite{HHT2} the $\sigma$-symmetry relates the connections $\wt \nabla^\lambda_t$ for different (i.e., complex conjugate) $\lambda$-values while the other symmetries preserve $\lambda$.  Therefore, we require only $\wt \nabla^\lambda_t$ to be equivariant with respect to $\delta$ and $\tau,$ i.e., $\wt \nabla^\lambda_T$ is gauge equivalent to
$\delta^*\wt \nabla^\lambda_t$ and $\tau^*\wt \nabla^\lambda_t.$

Moreover, the $\C^*$-family of flat connections $\wt \nabla^\lambda_t$ induces a $\C$-family of holomorphic structures $\wt \partial^\lambda_t$ on the topologically trivial complex rank $2$ bundle over $\C P^1$, since the underlying $(0,1)$-part
of the connections extends to $\lambda=0$ (see Proposition 4.5. in \cite{BoHeSch}). Let $V^\lambda_t = (V, \wt \partial^\lambda_t)$ denote the corresponding family of holomorphic bundles.
To obtain a Fuchsian DPW potential from $\wt \nabla^\lambda_t$, we need to show that $V^\lambda_t$ is holomorphically trivial for all $\lambda \in \D_{1+\epsilon}.$ 

As we have shown in Section \ref{Fuchsian}, the holomorphic structure on $V^\lambda_t$ is either $ \mathcal O\oplus \mathcal O,$ if the parabolic structure is semi-stable, or $ \mathcal O(-1)\oplus\mathcal O(1),$ if the parabolic structure is unstable. If the underlying holomorphic bundle $V^\lambda_t$ is trivial, then the gauge class of $\wt\nabla^\lambda_t$ can be represented by a Fuchsian system

\[\eta_t^\lambda = d+\sum_{k=1}^4A_k^t(\lambda)\frac{dz}{z-p_k},\] 
where $A_k^t(\lambda)\in\mathfrak{sl}(2,\C).$ 
Moreover,  being equivariant with respect to $\delta$ and $\tau$ translates into 
\begin{itemize}
\item $\delta$ symmetry:
\begin{equation}\label{delta-sym}\delta^*\eta_t^\lambda =D^{-1}\eta_t^\lambda D \quad \text{ with }\quad D=\matrix{i&0\\0&{-i}}\end{equation}
\item $\tau$ symmetry:
\begin{equation}\label{tau-sym}\tau^*\eta_t^\lambda = C^{-1}\eta_t^\lambda C\quad \text{ with }\quad C=\matrix{0&i\\i& 0 }\end{equation}

\end{itemize}
by Lemma \ref{lemma_modulipara}. By \cite[Theorem 4.7]{BoHeSch} the induced holomorphic structure at $\lambda=0$ is trivial, i.e.,  $V_0\cong \mathcal O\oplus\mathcal O.$
Moreover, by \cite{HHT2} and \cite{HHS} there exists $\alpha>0$ and $\varepsilon>0$ such that
for $t\in(0,\alpha)$ and $t \in (\tfrac{1}{4} - \alpha, \tfrac{1}{4}]$
 the bundle  
$V^\lambda_t$ is $\mathcal O\oplus\mathcal O$ for all $\lambda$ in the disc of radius $1+\varepsilon.$ The following main Theorem of this paper shows that  $\alpha \geq \tfrac{1}{4}.$\\

\begin{theorem}\label{thm:existence}
For all $t\in (0, \tfrac{1}{4}]$ there is a  Fuchsian DPW potential $\eta_t$, real analytic in $t$,
 $$\eta_t= d+\sum_{k=1}^4 A_k^t\frac{dz}{z-p_k}$$
defined on the 4-punctured sphere with the following properties:
\begin{itemize}
\item  there exists $\epsilon>0$ such that $\eta_t$ is well-defined for all $\lambda\in \mathbb D_{1+\epsilon}\setminus\{0\};$
\item the eigenvalues at each of the residues $A_k^t$ are $\pm t;$
\item the monodromy representation of $\eta_t$ is unitarizable for all $\lambda\in \mathbb S^1;$
\item the potential is symmetric, i.e., $\delta^*\eta=\eta.D,$ $\tau^*\eta=\eta.C.$
\end{itemize}
Moreover, the minimal surface $f^t$ corresponding to the potentials $\eta_t$ is the unique analytic continuation of the Plateau solution $f_t$ with respect to the geodesic polygon $\Gamma_t$. In particular, for $t=\tfrac{1}{2(g+1)},$ $ g\in\mathbb N^{\geq1}$,
the analytic continuation of $f^t$ is the Lawson surface $\xi_{1,g}$ of genus $g.$
\end{theorem}
\begin{proof}
We start with the Plateau solution for $\Gamma_t$ and consider the family of flat 
 logarithmic connections $\wt \nabla^\lambda_t$ constructed in Theorem \ref{1stnabla}.
The proof consists of two steps. We first show that the induced holomorphic bundle is $V_\lambda^t= \mathcal O \oplus \mathcal O$ for $\lambda\in \mathbb D_{1+\epsilon}\setminus\{0\}$ for an appropriate $\epsilon>0,$ and $t \in (0, \tfrac{1}{4})$ using continuity arguments. In the second step, we then show that we can gauge $\wt\nabla^\lambda_t$ into the symmetric normal form of Proposition \ref{lemma_modulipara} with holomorphic coefficients in $\lambda$ on $\mathbb D_{1+\varepsilon}\setminus\{0\}$,  i.e., such that there is no apparent singularities in $\lambda.$\\

{\bf Step 1:} For $t\sim 0$ and $t\sim \tfrac{1}{4}$ the bundle type $V_\lambda^t$ is shown to be  $\mathcal O\oplus\mathcal O$ for all $\lambda\in \mathbb D_{1+\varepsilon}$  
in \cite{HHT2} and \cite{HHS}, respectively. We want to show that the bundle type remains
$ \mathcal O\oplus\mathcal O$ for all $\lambda\in \overline{\mathbb D_{1}}$  and all $t\in(0,\tfrac{1}{4})$. Moreover, the bundle $V_{\lambda=0} = \mathcal O \oplus \mathcal O$ for every $t\in (0, \tfrac{1}{4}].$
Recall from Section \ref{sec:modulus} and Proposition \ref{up1extension} that, for every $t$ fixed, the function  $u_t\colon \C^* \rightarrow \C P^1,$
\[\lambda\in\C^*\mapsto  u_t(\lambda):= u(\wt\nabla_t^\lambda)\in \C P^1\]
is holomorphic, well-defined and extends holomorphically to $\lambda=0$ with value $u_t(0)=-1,$ as 
the underlying parabolic structure extends to $\lambda=0$ with a nilpotent strongly parabolic Higgs field. The stability of the parabolic structure at $\lambda=0$ follows as in \cite[Theorem 4.7]{BoHeSch}. Further, the holomorphic function $u_t$  is never constant, since otherwise
all connections $\wt\nabla^\lambda_t$ would be gauge equivalent for $\lambda\in S^1$ using the uniqueness part of the Mehta-Seshadri theorem.
Then  $\lambda\mapsto \wt\nabla^\lambda$ would be a constant map which contradicts, among others things,  the Sym-point condition
that the parabolic structure at the Sym-points must be strictly semis-stable while the parabolic structure is stable at $\lambda=0$. This implies that, for each $t$, the values of $\lambda \in\C$ for which $u_t(\lambda)=-1$ are discrete.
Due to the real analyticity of the Plateau solutions in $t$,
the map $t\mapsto \wt\nabla^\lambda_t$ is also real analytic in $t$. \\

By Proposition \ref{lemma_modulipara} logarithmic connections which are not Fuchsian are characterized by 
the property that
their underlying parabolic structure is unstable. For $t\sim 0$ and $t\sim \tfrac{1}{4}$, there are no
unstable parabolic structures inside the unit disc. By Mehta-Seshadri, logarithmic connections inducing unstable parabolic structures
are not unitarisable. Thus, by deforming $t$, unstable parabolic structures can not cross the unit 
$\lambda$-circle (where all connections must be unitarizable) in a continuous deformation.

Proposition \ref{up1extension} shows that the modulus for an unstable connection $\wt\nabla_t^\lambda$ must be $u =-1.$ Thus it remains to show that unstable structures cannot arise as limits of stable structures with $u=-1$ in our setup.
It is important to recall that the family of logarithmic connections $\wt \nabla^\lambda_t$ exist and is well-defined for all $\lambda\in \C^*$ and $t \in (0, \tfrac{1}{4}).$ Consider therefore a continuous sub family of $\wt \nabla^\lambda_t$ given by $t\in(a,b]\subset(0,\tfrac{1}{4})\mapsto\lambda_t\in\C$ such that
\[u_t(\lambda_t)=-1.\]
We want to show that 
if the parabolic structure of $\wt\nabla_t^{\lambda_t}$ is stable for all
$t<b,$ it is also stable at $t=b.$

Since $u_t$ is holomorphic in $\lambda$ (by \eqref{u1}, Proposition \ref{up1extension})  and non-constant, there exist a punctured disc around $\lambda_{t=b}$ such that $s_t = s(\wt \nabla^{\lambda_t}_t)$ is well-defined and holomorphic, see \eqref{nablaus}.
Therefore, the family of holomorphic functions
\[ r(t):=(u_t+1)\cdot s_t\]
 is continuous in $t$.
Since the parabolic structure of $\wt\nabla_t^{\lambda_t}$ is stable for $t<b$, $s_t$ is bounded for all $t<b$ and we obtain
\[((u_t+1)\cdot s_t)(\lambda_t)=0.\]
By continuity this also holds at $t=b, $ but
if the parabolic structure at $t=b$ would be unstable,
then \eqref{up1sexpansion} would give
\[0 = ((u_b+1)\cdot s_b)(\lambda_b)=1-4 b\]
contradicting $b < \tfrac{1}{4}.$
\begin{remark}
This conclusion heavily depends on the inequality $t<\tfrac{1}{4}.$ In fact, 
for the 2-lobed Delaunay tori 
at $t=\tfrac{1}{4}$,
there exists  $\lambda_s$ in the punctured unit disc with $u(\lambda_s)=-1$.
(The spectral parameter $\lambda_s$ is the unique branch value of the spectral curve of the Delaunay torus inside the punctured unit disc).
Then, as shown in \cite[Theorem 4.2]{HHS} in the setup of spectral curves, one can choose to deform the corresponding 
spectral data in direction of higher genus CMC surfaces in different ways by specifying to have either an unstable or a stable parabolic structure at 
$\lambda_s$.
\end{remark}

{\bf Step 2:} To show that the Fuchsian connections $\wt\nabla^\lambda$ can be parametrized with holomorphic coefficients in $\lambda$ on $\mathbb D_{1+\varepsilon} \setminus\{0\}$ with a simple pole at $\lambda=0$ we proceed as in the proof of \cite[Theorem 8]{He2}:

First assume that for every $\lambda_0 \in \mathbb D_1$ we can find an open neighborhood $U_{\lambda _0}\subset \C$ of the $\lambda-$plane such that $\wt\nabla^\lambda$ is locally around $\lambda_0$ gauge equivalent to a family of connections $\nabla^\lambda_{U_{\lambda_0}}$ with the desired symmetries (and holomorphic coefficients in $\lambda$) as in Lemma \ref{lemma_modulipara}. Since the closed disc $\mathbb D_1$ is compact, there exist finitely many points $\lambda_l \in \mathbb D_1$ and finitely many families $\nabla^\lambda_{U_{\lambda_l}}$ in the symmetric normal form such that $\mathbb D_1 \subset {\displaystyle \bigcup_{l} U_{\lambda_l} }$. On the intersection  $U_{\lambda_l} \cap U_{\lambda_k}$ the connections $\nabla^\lambda_{U_{\lambda_l}}$ and $\nabla^\lambda_{U_{\lambda_k}}$ are by construction gauge equivalent. This gives rise to transition functions 
$$g_{lk}\colon U_{\lambda_l} \cap U_{\lambda_k}\rightarrow GL(2, \C).$$

 Due to the uniqueness in Lemma \ref{lemma_modulipara} the image of $g_{lk}$ lies in fact the subgroup generated by the matrices $C$ and $D.$ Moreover, these transition functions $\{g_{lk}\}$ defines a cocycle on $\mathbb D_{1+\epsilon}.$ This cocycle is integrable, as $\mathbb D_{1+\epsilon}$ is simply connected, which gives rise to the desired DPW potential on all $\mathbb D_{1+\epsilon}.$ \\

It remains to show the existence of  $\nabla^\lambda_{U_{\lambda_0}}$ for every $\lambda_0 \in \mathbb D_1.$
If $\wt \nabla^{\lambda_0}$ is irreducible,  this follows from 
the uniqueness part of Lemma
\ref{lemma_modulipara} and the fact that the stabiliser of an irreducible connection is $\pm$Id.
At $\lambda=0$ the existence of $\nabla^\lambda_{U_0}$ follows from the fact that the induced holomorphic structure is stable and the arguments for irreducible connections in the proof of Lemma  \ref{lemma_modulipara}  carries over verbatim.

 For the (finitely many) $\lambda_0$ for which $\wt\nabla^{\lambda_0}$ is reducible, i.e., where the connection is given by the direct sum of line bundle connections, we proceed as follows:  consider the 4-fold covering $\pi \colon \Sigma \to\C P^1$ by taking the quotient with respect to $\delta$ and $\tau$. Using the push-forward construction  (see \cite{Biswas})
one obtains logarithmic connections on a 4-punctured sphere with local eigenvalues
$\pm\tfrac{1}{4}$ at 3 of the 4 singular points, and $\pm t$ at fourth singular point (which is the image of
$p_1,\dots,p_4$ under $\pi$). On the quotient, as $t\in(0, \tfrac{1}{4}]$, there are no reducible connections, see \cite[Biswas conditions]{HH_abel}. Therefore, we can proceed as in the irreducible case on the quotient to  obtain $\wh \nabla^\lambda_{U_{\lambda_0}}$ around $\lambda_0.$ Pulling back the connections $\wh \nabla^\lambda_{U_{\lambda_0}}$  by the 4-fold covering to $\Sigma$, and desingularising at the apparent singularities corresponding to the eigenvalues $\pm \tfrac{1}{4}$ yields the desired $ \nabla^\lambda_{U_{\lambda_0}}$  of the form of Lemma \ref{lemma_modulipara}, well-defined around $\lambda_0.$  
\end{proof}

We end the paper by specifying the symmetric DPW potential $\eta$ for the Lawson surfaces $\xi_{1,g}$. Following \cite[Section 2.1]{HHT2}, $\eta$ is of the form
\[\eta = \left(
\begin{array}{cc}
 -\frac{4 a z}{z^4+1} & \frac{2 \sqrt{2} \left(b \left(z^2-1\right)-c \left(z^2+1\right)\right)}{z^4+1} \\
 \frac{2 \sqrt{2} \left(b \left(z^2-1\right)+c \left(z^2+1\right)\right)}{z^4+1} & \frac{4 a z}{z^4+1} \\
\end{array}
\right)dz\]
for holomorphic functions $a,b,c\colon \mathbb D_{1+\epsilon}^*\to\C$
with first order poles at $\lambda=0$, 
and
\[a^2-b^2-c^2=-t^2\]
with $t=\tfrac{1}{2g+2}.$
Moreover, the residue of the DPW potential at $\lambda=0$ is nilpotent (and non-zero if 
$t>0$) which is equivalent to
\[\text{Res}_{\lambda=0}\left (-\frac{b^2}{c^2} \right)=-1.\]
At $t=0$ we have by \cite{HHT2} 
\[a=0=b=c.\]
 In \cite[Section 5]{HHT2} we have given an iterative algorithm to compute the Taylor expansion of $a_t(\lambda)$, $b_t(\lambda)$ and $c_t(\lambda)$  in $t$  at $t=0$. It turns out that the $n$-th $t$-derivative  of $a_t,$ $b_t$ and $c_t$ at $t=0$ are Laurent polynomials in $\lambda$ obtained from solving (non-degenerate) finite-dimensional linear systems with coefficients given by multi-polylogarithms. 
For example, the first order derivatives are given by 
\begin{equation}
\begin{split}
\dot a&=\tfrac{1}{2}(\lambda^{-1}-\lambda)\\
\dot b&=\dot c=-\tfrac{1}{2\sqrt{2}}(\lambda^{-1}+\lambda).\\
\end{split}
\end{equation}
At $t=\tfrac{1}{4}$ the potential can be computed explicitly 
in terms of elliptic functions.


\begin{thebibliography}{9}


\bibitem{Biq} O. Biquard, {\em Fibr\'es paraboliques stables et connexions singuli\`eres plates}, 
Bull. Soc. Math. Fr. {\bf 119} (1991), 231--257.

\bibitem{Biswas}  I. Biswas, {\em Parabolic bundles as orbifold bundles}, Duke Mathematical Journal, 88(2), 305--326,
(1997).

\bibitem{Bobenko} A. I. Bobenko, {\em  All constant mean curvature tori in $\R^3$, $\S^3$ ,$H^3$ in terms of theta-functions,} Math. Ann. 290 (1991), no. 2, 209--245.

\bibitem{BoHeSch} A. I. Bobenko, S. Heller, N. Schmitt, {\em Constant mean curvature surfaces based on fundamental quadrilaterals}, to appear in: Mathematical Physics, Analysis and Geometry; arXiv:2102.03153.


\bibitem{De}  P. Deligne, {\em Equations diff\'erentielles \`a points
singuliers r\'eguliers}, Lecture Notes in Mathematics, Vol. 163, Springer-Verlag,
Berlin-New York, 1970.


\bibitem{DPW}
J.~Dorfmeister, F.~Pedit, and H.~Wu, \emph{Weierstrass type representation of
  harmonic maps into symmetric spaces}, Comm. Anal. Geom. \textbf{6} (1998),
  no.~4, 633--668.
  
   \bibitem{HH_abel}
L.~Heller, S.~Heller, \emph{Abelianisation of Fuchsian systems and Applications}, Journal of
Symplectic Geometry,
Volume 14, Number 4, 1059--1088, (2016)
 
    \bibitem{FCS} D. ~Fischer-Colbrie, R~ Schoen,  \emph{The structure of complete stable minimal surfaces in 3-manifolds of nonnegative scalar curvature.} Comm. Pure Appl. Math. 33 (1980), no. 2, 199-211.
  
  \bibitem{HHS}
L.~Heller, S.~Heller, N.~Schmitt, \emph{Navigating the space of symmetric
  {CMC} surfaces}, J. Differential Geom. \textbf{110} (2018), no.~3, 413--455.
  
    \bibitem{HHT}
L.~Heller, S.~Heller, M.~Traizet, \emph{Area estimates for high genus Lawson surfaces}, to appear in J. Differential Geom., arxiv:1907.07139.

  \bibitem{HHT2}
L.~Heller, S.~Heller,  M.~Traizet, \emph{Complete families of embedded high genus CMC surfaces in the $3$-sphere}, arXiv:2108.10214.

\bibitem{He1}
 S.~Heller, {\em Lawson's genus two minimal surface and meromorphic connections},
 Math. Z., Volume 274 (2013), pp 745--760.
 
 \bibitem{He2}
 S.~Heller,  {\em A spectral curve approach to Lawson symmetric CMC surfaces of genus 2}, Math. Annalen, Volume 360, Issue 3 (2014), pp 607--652.  
 
\bibitem{Hitchin}
N. Hitchin:
{\em Harmonic maps from a 2-torus to the 3-sphere.}
J. Differential Geom. 31 (1990), no. 3, 627--710.

\bibitem{Kap}  N. Kapouleas: {\em Minimal surfaces in the round three-sphere by doubling the equatorial two-sphere, I.}, J. Differential Geom. 106 (2017), no. 3, 393--449. 

\bibitem{KapYan}
N. Kapouleas,  S. D. Yang:
{\em Minimal surfaces in the three-sphere by doubling the clifford torus},
Amer. J. of Math. 132 (2010), no. 2, 257--295. 

\bibitem{KapWiyStability}
N. Kapouleas, D. Wiygul:
{\em The index and nullity of the Lawson surfaces $\xi_{g,1}$.}, Camb. J. Math. 8 (2020), no. 2, 363--405.


\bibitem{KPS}
H. Karcher, U. Pinkall, I. Sterling, {\em New minimal surfaces in $\mathbb S^3$}, J. Differential Geom., Volume 28, Number 2 (1988), 169--185.




\bibitem{KW} S. Kim, G. Wilkin,
{\em Analytic convergence of harmonic metrics for parabolic Higgs bundles},
 Journal of Geometry and Physics, 127:55-67, 2018.


\bibitem{Kusner}
R. Kusner, {\em Comparison surfaces for the Willmore problem}, Pacific J. Math. 138 (1989), no. 2, 317--345.


\bibitem{KLS} E. Kuwert; Y. Li; R. Sch\"atzle, {\em The large genus limit of the infimum of the Willmore energy}, Amer. J. Math. 132 (2010), no. 1, 37--51.


\bibitem{Lawson}
H. B. Lawson, {\em  Complete minimal surfaces in $S\sp{3},$}  Ann. of Math. (2) 92 (1970), 335--374.

\bibitem{LS} F. Loray,  M.-H. Saito, {\em  Lagrangian Fibrations in Duality on Moduli Spaces of Rank 2 Logarithmic Connections Over the Projective Line}, IMNR, Volume 2015, Issue 4.


\bibitem{MN}
F.C. Marques, A. Neves, {\em Morse index of multiplicity one min-max minimal hypersurfaces}, Adv. Math. 378 (2021) Paper No. 107527.



\bibitem{MS} V. B. Mehta and C. S. Seshadri, {\em Moduli of vector bundles on curves with parabolic structures},
Math. Ann. {\bf 248} (1980), 205--239. 

\bibitem{Men} C. Meneses, {\em Remarks on groups of bundle automorphisms over the Riemann sphere}, Geom. Dedicata, 196(1):63--90, 2018.


\bibitem{PS} U. Pinkall, I. Sterling, {\em  On the classification of constant mean curvature tori,} Ann. of Math. (2) 130 (1989), no. 2, 407--451.


 \bibitem{Pir} G. Pirola,  {\em Monodromy of constant mean curvature surface in hyperbolic space},  Asian J. Math.  11, no. 4, 651--669 (2007). 
 

  
\bibitem{Sim0} C. Simpson, {\em Harmonic bundles on noncompact curves}, J. Am. Math. Soc., 3(3), 713--770, 1990.
  
  






\bibitem{Wu}
H. Wu: {\em A simple way for determining the normalized potentials for harmonic maps}, Ann. Global
Anal. Geom. 17 (1999), 189--199.
\end{thebibliography}
\end{document}